\documentclass[11pt,article]{article}

\usepackage{amsmath,amssymb,float,array,subcaption,amsthm,enumerate}
\usepackage[usenames,dvipsnames]{color}
\usepackage{verbatim}
\usepackage[normalem]{ulem}
\usepackage{setspace,xspace}

\newcommand{\dsum}{\displaystyle\sum}

\newcommand{\tsum}{\textstyle\sum}

\newcommand{\xu}{\underline{x}}
\newcommand{\xo}{\overline{x}}
\newcommand{\xL}{\underline{x}_k^o}
\newcommand{\SPR}{SPR compatible\xspace}

\newcommand{\PBS}[1]{\let\temp=\\#1\let\\=\temp}
\newlength{\tmplength} 
\newtheorem{lemma}{Lemma}
\newtheorem{theorem}{Theorem}
\newtheorem*{theorem*}{Theorem}
\newtheorem{definition}{Definition}
\newtheorem{assumption}{Assumption}

\title{\bfseries \Huge A more general Pandora's rule?}

\author{Wojciech Olszewski\thanks{Department of Economics,
    Northwestern University, USA} \and Richard Weber\thanks{Statistical
    Laboratory, University of Cambridge, UK}}

\date{}

\begin{document}

\maketitle

\begin{abstract} 
\indent In a classic model analysed by Weitzman an agent is
presented with boxes containing prizes. She may open boxes in any
order, discover prizes within, and optimally stop. She wishes to
maximize the expected value of the greatest prize found, minus costs
of opening boxes. The problem is solved by a so-called Pandora's
rule, and has applications to searching for a house or job. However,
this does not model the problem of a student who searches for the
subject to choose as her major and benefits from all courses she
takes while searching.

So motivated, we ask whether there exist any problems for which a
generalized Pandora's rule is optimal when the objective is a more
general function of all the discovered prizes. We show that if a
generalized Pandora's rule is optimal for all specifications of
costs and prize distributions, then the objective function must take
a special form. We also explain how the Gittins index theorem can
be applied to an equivalent multi-armed bandit problem to prove
optimality of Pandora's rule for the student's problem. However, we
also show that there do exist some problems which are not of
multi-armed bandit type for which Pandora's rule is optimal.
  \medskip



\end{abstract}


\section{Weitzman's problem and its generalization}

\subsection{Hunting the best prize}
In a classic problem that was first analyzed by Weitzman 
\cite{Weit79} an agent called Pandora is presented with $n$ boxes,
each of which contains a prize. Pandora can, by paying a known cost
$c_{i}$, open box $i$ to reveal its prize. The nonnegative value of
the prize, denoted $x_{i}^{o}$, is not known until the box is opened,
but ex ante it has known distribution $F_{i}$. The superscript `o' is
provided as a mnemonic for `opened' or `observed'. Pandora wishes to
choose the order of opening the boxes, and when to stop opening, so as
to maximize the expected value of the greatest discovered prize, net
of the sum of the costs paid to open boxes. Weitzman's problem is
attractive for two reasons.  Firstly, it has an enormous number of
applications, such as to searching for a house, job, or research project
to conduct. A key feature is that it combines problems of scheduling
(in what order should the boxes be opened?) and stopping (when should
one be content to take the greatest prize found thus far?)

The second reason that the problem is attractive is that it has a remarkably
simple solution, which we now describe. Suppose there is just one unopened
box, say box $i$. However, there is also a {\itshape reservation prize}
already on the table, of value $y$ that may be taken at any time. It is
optimal not to open the box $i$ if and only if 
\begin{equation*}
y\geq -c_{i}+E\max [y,x_{i}^{o}].
\end{equation*}%
The expectation is taken over $x_{i}^{o}$ according to the
distribution $F_{i}$. The inequality is equivalent to $c_{i}\geq E\max
[0,x_{i}^{o}-y]$, whose right-hand is decreasing in $y$. So there is a
least nonnegative $y$ for which is true:
\begin{equation}
x_{i}^{\ast }=\min \{y:c_{i}\geq E\max [0,x_{i}^{o}-y],\ y\geq 0\},
\label{windex}
\end{equation}
and $x_{i}^{\ast }$ is called the \emph{reservation value} (or
reservation prize) of $%
x_{i}$.

The so-called Pandora's rule, which is optimal for Weitzman's problem,
is to first compute the reservation value of each box, as if each were
the only box, and then open boxes in descending order of these values
until a prize is found whose value exceeds the reservation value of
any unopened box.

Attractive as it is, the Weitzman model does not cover an important
and large class of applications in which the agent's utility is not
merely a function of the one prize the agent takes at the end of
search, but of all prizes uncovered. Such problems contain features of
both Weitzman's problem and the celebrated multi-armed bandit problem
as solved by Gittins and Jones \cite{Git74}. For example, a student
benefits from the courses she takes while searching for the subject to
choose as major; or people obtain a flow utility of dating with
different partners in the process of looking for a spouse; or an
institution which experiments with different forms of organization,
before adopting a more permanent form, is affected by those temporary
forms\footnote{Weitzman anticipated that Pandora's rule would not
  generalize to such problems. He wrote: ``\emph{If
    some fraction of its reward can be collected from a research
    project before the sequential search procedure as a whole is
    terminated, that could negate Pandora's rule in extreme
    cases}.''
However, he gave no supporting detailed analysis, and it turns
out to be difficult to say whether or not some interesting
generalization might be possible.}.

Motivated by this, and other applications to be described in the
following sections, we now consider a possible generalization of
Weitzman's model. Suppose that $S$ is the set of opened boxes at the
point we stop, and the vector of the prize values found is
$x_{S}^{o}=(x_{i}^{o},\ i\in S)$.  In Weitzman's problem the aim is to
maximize the expected value of
\begin{equation}
R(x_{S}^{o})=\max_{i\in S}x_{i}^{o}-\sum_{i\in S}c_{i}.  \label{eq:13}
\end{equation}%
(Strictly speaking, $R$ is a function of $S$ and $x_S^o$.  However, in
\eqref{eq:13}, and in similar definitions that depend on $(S,x_S^o)$,
it is convenient to suppress the $S$.)  Now consider a more general
reward, expressed as a utility that depends on all the prizes
discovered.
\begin{equation}
R(x_{S}^{o})=u(x_{S}^{o})-\sum_{i\in S}c_{i}.  \label{eq:17}
\end{equation}
We shall use the notation $\varnothing$ to denote both the empty set
and the empty vector; so may write $R(\varnothing)=0$.  The paper is
primarily concerned with the following question.

\paragraph{Question 1}

{\itshape For what utility functions $u$ is a simple (generalized) Pandora's
rule optimal?}

\subsection{A generalized Pandora's rule}

To answer Question~1 we must start by saying what a `generalized
Pandora's rule' might be. Suppose a set of boxes $S\subset
N=\{1,\dots ,n\}$ has been opened, and $i\not\in S$. We might ask,
{\itshape what is the smallest prize whose addition to the set of
  prizes already discovered would cause it to be optimal to stop
  rather than open box $i$ and then stop?} This defines a reservation
value (or prize) for $x_{i}$, of
\begin{equation}
x_{i}^{\ast }=\min \{y:u(x_{S}^{o},y)\geq -c_{i}+Eu(x_{S}^{o},y,x_{i}^{o}),\
y\geq 0\}  \label{eq:20}
\end{equation}%
with the expectation being taken over $x_{i}^{o}$. Notice that for
this to make sense we must assume, as we now do, that $u$ is a
function that maps a vector of any length to a real value. Note that
if $x_{S}^{o}$ is a vector of length $n-1$ then the right-hand side of
\eqref{eq:20} requires $u$ to be defined over vectors of length $n+1$.
It is straightforward to see that for Weitzman's problem
  \eqref{windex} and \eqref{eq:20} coincide, as do the following 
  definitions of Weitzman's Pandora's rule and a generalized Pandora's
  rule.

\begin{definition}[{Weitzman's Pandora's rule}]
\emph{Open the unopened box with greatest reservation value, as
  defined by \eqref{windex}, until there is no unopen box whose
  reservation value exceeds the greatest prize that has been found.}
\end{definition}
\begin{definition}[{Generalized Pandora's rule}] \emph{Open the unopened
    box with greatest reservation value, as defined by \eqref{eq:20},
    until there is no unopened box whose reservation value exceeds 0.}
  \medskip
\end{definition}

Notice that $x_i^*$ is not an index in the most usual sense. Unlike
the reservation prize \eqref{windex} in Weitzman's problem, or a
Gittins index, $x_i^*$ is a function not only of of $c_i$ and $F_i$,
but also $x_S^o$, the vector of values of prizes that have already been
uncovered. Consequently, after opening a box, the reservation values
of all the unopened boxes must be recomputed from \eqref{eq:20}.

\subsection{A generalized utility function}

We cannot hope that the generalized Pandora's rule should be optimal
unless we place some restrictions on the utility function $u$. What
minimal constraints should we impose on $u$ to obtain a nice
answer? Let us take as a guide the fact that the utility
$u(x_1,\dotsc,x_k)=\max_i x_i$ of Weitzman's problem has several
special properties, which we now group under the heading of Assumption
1 and will also wish to require subsequently.

\begin{assumption}\ \\[-16pt] {\itshape
    \begin{itemize}
\item $u(\varnothing)=0$ and $u(0,x_2,\dotsc,x_k)=u(x_2,\dotsc,x_k)$.

\item $u$ is continuous, nonnegative, symmetric, nondecreasing and
submodular in its arguments;
\end{itemize}}
\end{assumption}
By `symmetric' we mean that the value of $u$ for any $k$-tuple of arguments
is the same as its value for any permutation of that $k$-tuple. So utility
depends only the set of prizes found, not on the order in which they are
found. By `submodular', we mean that for any vectors $x$ and $y$ of the same
length 
\begin{equation*}
u(x)+u(y)\geq u(x\wedge y)+u(x\vee y)
\end{equation*}
where $x\wedge y$ and $x\vee y$ denote the minimum and maximum of $x$
and $y$ taken component-wise. An equivalent statement is that the
increase in $u(x)$ obtained by increasing one component of $x$ becomes
no greater as any other component becomes greater. That is, for any
$x_S^o$, $\xu_1<\overline{x}_1$ and $\xu_2<%
\overline{x}_2$,

\begin{equation*}
u(x_{S}^{o},\overline{x}_{1},\overline{x}_{2})-u(x_{S}^{o},\underline{x}_{1},%
\overline{x}_{2})\leq u(x_{S}^{o},\overline{x}_{1},\underline{x}%
_{2})-u(x_{S}^{o},\underline{x}_{1},\underline{x}_{2})\text{.}
\end{equation*}

We shall assume throughout the rest of the paper that Assumption 1 holds. It is
helpful to remember that $u$ is symmetric when reading some of
the expressions below.

There is a further property of \eqref{eq:13}, which we now state as 
distinct from Assumption 1.

\begin{assumption} {\itshape The benefit of increasing a component of
    $x$ from $0$ to some positive value $x_{i}^{o}$ is independent of
    the values of other components of $x$ which are greater than
    $x_{i}^{o}$. That is, for $x_{S}^{o}$ and any $x_{i}\leq
    \underline{x}_{j}<\overline{x}_{j}$, with $i,j\not\in S$,
\begin{equation}
u(x_{S}^{o},\overline{x}_{j},x_{i})-u(x_{S}^{o},\overline{x}%
_{j},0)=u(x_{S}^{o},\underline{x}_{j},x_{i})-u(x_{S}^{o},\underline{x}%
_{j},0).  \label{eq:15}
\end{equation}}
\end{assumption}

At first sight this assumption appears very strong, but we shall see
that if a generalized Pandora's rule is to be optimal for all choices
of $c_{i}$ and $F_{i}$, then it is necessary.\smallskip

Consider now the following special case of utility. Since it plays an
important role in our paper we define a name for it.  In the proof of
Theorem 2 below, we see that this form of $u$ poses a problem that is
equivalent to a multi-armed bandit problem.

\begin{definition}\label{defSPR}
  A utility function $u$ is said to be `Strongly Pandora's Rule
  compatible' (SPR) if
  \begin{equation}
    u(x_{S}^{o})=u(\max_{i\in S}x_{i}^{o})-f(\max_{i\in S}x_{i}^{o})+\sum_{i\in
      S}f(x_{i}^{o})  \label{eq:18}
  \end{equation}%
  where $f$ denotes some function, and then $u,f,u-f$ are all
  nonnegative and nondecreasing functions, and $u(0)=f(0)=0$.
\end{definition} 
It is straightforward to check that if $u$ is \SPR then it satisfies
Assumptions 1 and 2.  Moreover, if $u$ is \SPR then the generalized
reservation value is equal to Weitzman's reservation value.  To see
this, we note that the reservation value of $x_i$ is the least $y$
such that
\begin{align}
  u(x_{S}^{o},y,0)& \geq -c_i +
  Eu(x_{S}^{o},y,x_i^o),
\end{align}
or equivalently, with $x_1=\max_{j\in S}x_j^o$, the least $y$ such that
\begin{equation}
  \begin{array}{lll}
u(x_1)&\geq -c_i+E \max\{u(x_1)+f(x_i^o),u(x_i^o)+f(x_1)\},&\quad  y\leq x_1,\\[8pt]
 u(y)&\geq -c_i+E \max\{u(y)+f(x_i^o),u(x_i^o)+f(y)\},&\quad  y>x_1.\\
  \end{array}
\label{eq:21}
\end{equation}
or, in a form that we wish to compare later to \eqref{eq:9},
\begin{equation}
  \label{eq:8}
  u(y)-f(y)\geq -c_i +Ef(x_i^o)+ E \max\{u(y)-f(y),u(x_i^o)-f(x_i^o)\},\quad  y>x_1.
\end{equation}
In the special case of Weitzman's problem $f=0$ and $u(x)=x$, and then
the second line in \eqref{eq:21} (or \eqref{eq:8}) reduces to $y\geq
-c_i+E\max\{x_i^o,y\}$, which agrees with the calculation of
Weitzman's reservation prize value in \eqref{windex}.


Throughout the paper we assume Assumption 1 holds. In Section~2, we
show if the generalized Pandora's rule is to be optimal for all
possible choices of $((c_{i},F_{i}),\ i\in N)$ then the utility $u$
must be \SPR.  That is, if $u$ satisfies Assumption 1, but is not
\SPR, then there exist some costs $(c_{i},\ i\in N)$ and prize
distributions $(F_{i},\ i\in N) $ such that Pandora's rule is not
optimal.

In Theorem~\ref{thm:2} of Section~3 we present a converse to the
above: namely, that if $u$ is \SPR then Pandora's rule is optimal
for all $((c_{i},F_{i}),\ i\in N)$.  We prove Theorem~\ref{thm:2} by
recasting the problem as an equivalent multi-armed bandit problem and
applying the Gittins index theorem.  The connection between the
multi-armed bandit problem and Weitzman's problem has been previously
noticed by Chad and Smith \cite{Chade06} who remarked that
``Weitzman's method is a nice application of Gittins' solution of the
bandit problem''.  However they focused upon a problem of simultaneous
search, rather than Weitzman's sequential search, and did not actually
explain how the solution to Weitzman's problem can be obtained from
the Gittins index theorem.

Of course, the results of Section 2 leave open the possibility that
there might exist interesting problems in which $u$ is not SPR
compatible, and yet Pandora's rule is optimal for some, but not
all, $((c_{i},F_{i}),\ i\in N)$.  In Theorem~\ref{thm:3} of Section 4
we establish that this happens if Assumptions 1, 2 and a further
Assumption 3 (which we also denote as `ORD') are all
satisfied. Theorem 3 is stronger than Theorem 2 because if $u$
is \SPR then ORD is satisfied. Theorem 3 cannot be proved by applying
the Gittins index theorem {or by some adaptation of Weiztman's proof}.
We now conclude this introduction by describing a problem of this type
to which Theorem~\ref{thm:3} provides the optimal strategy.

\paragraph*{Example 1.}
Suppose that each $F_{i}$ is the two-point distribution such that $%
x_{i}^{0}=0$ or $x_{i}^{0}=1$ with given probabilities $q_{i}$ and $%
p_{i}=1-q_{i}$, respectively. One has in mind that each box may or may
not contain a prize, but all prizes are the same. Let $\psi $ be a
concave increasing function of the total value of prizes found and
consider the objective function
\begin{equation}
R(x_{S}^{o})=\psi \left( \textstyle\sum_{i\in S}x_{i}^{o}\right) -\textstyle%
\sum_{i\in S}c_{i}.  \label{eq:19}
\end{equation}%
This $u=\psi $ does not obey Assumptions 1 and 2, so by Theorem 1 
Pandora's rule is not optimal for all $c_{i}$ and $F_{i}$. However,
for the form of $F_{i}$ given, we will see that it is easy to check
that the sufficient conditions of Theorem~\ref{thm:3} are met, and so
we may conclude that Pandora's rule is optimal. The rule takes the
following form. Suppose $c_{1}/p_{1}\leq \cdots \leq
c_{n}/p_{n}$. Then one should open boxes in the order $1,2,\dotsc $,
but stop when we are about to open some box $j$, have thus far found
$k\leq j-1$ prizes, and
\begin{equation*}
\psi (k)\geq -c_{i}+p_{i}\,\psi (k+1)+q_{i}\,\psi (k),
\end{equation*}%
or equivalently, 
\begin{equation*}
c_{i}/p_{i}\geq \psi (k+1)-\psi (k).
\end{equation*}%
While it would be possible to guess and establish this fact by a fairly
short tailored proof, using induction on $n$, the sufficient conditions
provided by Theorem~\ref{thm:3} are quick to check.

\section{Necessary conditions for strong Pandora's rule optimality}

We state some preliminary lemmas, whose proofs are in Appendices A and
B.
\begin{lemma}
  Suppose the utility $u$ satisfies Assumption 1 and Pandora's rule
  maximizes expected utility for all costs $c_{i}$ and distributions
  $F_{i}$. Then $u$ also satisfies Assumption 2.
\end{lemma}

\begin{lemma}
  Suppose the utility $u$ satisfies Assumptions 1 and 2.  For any
  $(x_S^o:\,i\in S)$, let $x_{\ell}$ denote the $\ell$th greatest
  element. Then, 

  (a) there exist functions $f_{\ell}:\mathbb{R}\rightarrow
  \mathbb{R}$, $\ell=1,2,\dotsc$, such that for any $x_S^o$ we have
\begin{equation}
u(x_{S}^{o})=\dsum\limits_{\ell=1}^{|S|}f_{\ell}(x_{\ell}),  \label{4}
\end{equation}

(b) $f_{\ell}(x)$ is (weakly) increasing in $x$ and
(weakly) decreasing in $\ell$,\medskip

(c) $f_{\ell}(x)-f_{\ell+1}(x)$ is weakly increasing in $x$.
\end{lemma}

The main theorem of this section now follows. The fact that it
concerns conditions which are necessary if Pandora's rule is to be
optimal for all costs $c_{i}$ and distributions $F_{i}$ is signaled by
the word `strong' that appears in the title to this
section. Similarly, `strong' appears in the title to Section 3, but
not in the title to Section 4, which concerns sufficient conditions,
depending upon specific $c_i$, $F_i$, for which Pandora rule is
optimal.

\setcounter{theorem}{0}
\begin{theorem}\label{thm:1}
  Suppose the utility $u$ satisfies Assumption 1, and Pandora's rule
  maximizes expected utility for all costs $c_{i}$ and distributions
  $F_{i}$.  Then necessarily $u$ must be \SPR
  (Definition~\ref{defSPR}).  
\end{theorem}

\paragraph{Remark.} Note that $u(x_1,x_1)=u(x_1)+f(x_1)\implies
f(x)=u(x,x)-u(x)$.  As in Lemma 2, we let $x_1\geq x_2\geq\cdots\geq
x_{|S|}$ denote the ordered $(x_i^o:i\in S)$.  Then we may write
\eqref{eq:18}, as
  \begin{align}
    u(x_{S}^{o})&=u(x_1)+\dsum\limits_{\ell=2}^{|S|}f(x_{\ell}),\label{5b}
  \end{align}
where $f(x)=u(x,x)-u(x)$. 

\begin{proof}[Part proof of Theorem 1]
  We give in Appendix C the proof that $u$ must satisfy \eqref{eq:18}.
  For now, we prove that if \eqref{eq:18} holds then the subsequent
  statements within Definition 3 are true.  Clearly, $u(0)=f(0)=0$,
  and $u,f$ are nonnegative. The other facts are proved as follows.

\begin{itemize}
\item $f(x)$ is nondecreasing in $x$, since  for $x<x'$
  \begin{align*}
f(x)=    u(x,x)-u(x)&= u(x,x')-u(x')\leq u(x',x')-u(x'),
\end{align*}
where the equality is by Assumption 2, and the inequality is by Assumption 1.
\item $u- f$ in nonnegative since by an application of Assumption 1 (submodularity)
\begin{align*}
u(x)-f(x)&=u(x)-[u(x,x)-u(x,0)]\geq u(x)-[u(0,x)-u(0,0)]=0.
\end{align*}
\item $u(x)-f(x)$ is nondecreasing in $x$, since for $x<x'$
  \begin{align*}
    u(x)-f(x)&=u(x)-[u(x,x)-u(x)]\\
&\leq u(x)-[u(x',x)-u(x')]\\
&\leq u(x')-[u(x',x')-u(x')]\\
&=u(x')-f(x'),
  \end{align*}
  where the first and second inequalities follow from Assumption 1,
  by submodularity and monotonicity, respectively.\qedhere
\end{itemize}
\end{proof}

\section{Sufficient conditions for strong Pandora's rule optimality}

Theorem 2 is a converse to Theorem 1.

\begin{theorem}\label{thm:2}
If utility $u$ is \SPR then, for all costs $c_i$
and distributions $F_i$, Pandora's rule is optimal.
\end{theorem}

\begin{proof}

  We prove this by mapping the problem to an instance of a multi-armed
  bandit problem and applying the Gittins index theorem.  In
  particular, we use the fact that the Gittins index theorem is true
  for undiscounted target processes (Gittins, Glazebrook and Weber\
  \cite{Git11}, Chapter 7). In a problem about target processes,
  reward accrues only until the first time that one of the processes
  reaches a target (or certain state). This is also the problem that
  Dumetriu, Tetali and Winkler \cite{Dum03} cutely call `playing golf
  with more than one ball', in which the aim is to minimize costs (or
  maximize reward) until one of several golf balls is first sunk in a
  hole.

  Consider a family of $n$ alternative bandits processes, each
  evolving on its own state space. We will think of the covered
  variables as bandits in state 0, and the target or hole as a certain
  state 1. Bandit $i$ starts in its initial state $0$. When it is
  continued for the first time a reward $-c_i$ accrues and the state
  makes a random transition to a new state which we denote as
  $(x_i^o,1)$, where $x_i^o$ is chosen according to distribution
  $F_i$. When the bandit is continued from this state a reward
  $f(x_i^o)$ accrues and the state makes a deterministic transition to
  $(x_i^o,2)$. When the bandit is continued from this state a reward
  $u(x_i^o)-f(x_i^o)$ accrues and the state makes a deterministic
  transition to state $1$.  As this state is the target the problem
  comes to an end. So only one of the bandits can make a third
  step. (Note that the Gittins index theorem is true for an
  uncountable state space, as we may have here.)

  It is clear that in maximizing reward it will be optimal to continue
  some set of bandits twice each, and then one of these a third time
  (the one we choose to be the one to enter state 1 and bring the
  problem to its end). If we were to continue bandits $1,\dotsc,k$
  twice, and then pick $i$ for continuation a third time, we would
  achieve reward
\[
  -(c_1+\cdots+c_k)+f(x_1^o)+ f(x_2^o)+\cdots+
  f(x_k^o)+u(x_i^o)-f(x_i^o),
  \]
  which is the same objective function as we seek to minimize when $u$
  is \SPR. Since $u-f$ is nondecreasing the process we should choose
  for continuation a third time is clearly the one having greatest
  uncovered value $x_i^o$.

  Having mapped our problem to an instance of a multi-armed bandit
  problem we can appeal to the Gittins index theorem, which says that
  expected total reward is maximized by continuing at each stage a bandit
  with the greatest Gittins index (with ties broken
  arbitrarily).  The Gittins index of bandit $i$ in this problem can
  be found by the `calibration method' which computes the Gittins
  index as the least $\lambda$ such that
  \begin{equation}
    \lambda\geq -c_i +  E\max\left\{f(x_i^o)+\lambda,
      f(x_i^o)+[u(x_i^o)-f(x_i^o]
    \right\}.\label{eq:9}
  \end{equation}
  The right hand side of \eqref{eq:9} is the maximum expected reward
  which can be obtained by continuing bandit $i$ once, and thereafter
  either continuing it once more to take reward $f(x_i^o)$ and then
  retiring with reward $\lambda$, or continuing it twice more until
  reaching state 1.  Comparing \eqref{eq:9} to \eqref{eq:8}, and
  identifying $\lambda$ with $u(y)-f(y)$ (which is nondecreasing in
  $y$) we see that the Gittins indices do indeed prescribe exactly the
  same policy as does Pandora's rule for generalized reservation
  values and a $u$ that is \SPR.
\end{proof}

\section{Sufficient conditions for Pandora's rule optimality}

We require a new
assumption.

\paragraph{Assumption 3 (ORD).}(\emph{History-Independence of the
Ordering of Reservation Values}):\\
{\itshape The ordering of reservation values $x_{k}^*$ of the covered variables
is independent of both the number of variables that have already been
uncovered and their realizations. That is, for any $S$, $x_{S}^{o}$,
and $k,j\notin S$,
\[
x_{k}^*(x_{S}^{o})\geq x_{j}^*(x_{S}^{o})\quad\iff\quad x_{k}^*(\varnothing )\geq x_{j}^*(\varnothing )\text{.\bigskip } 
\]}

We denote this property by the abbreviation ORD. Admittedly, it is
very strong assumption.  Unlike Assumptions 1 and 2, this assumption
is not simply a property of the utility function alone, but a joint
property of the utility function $u$, and the $(c_i,F_i)$. In most
problems it is easy to check whether or not ORD is satisfied. In particular,
it is satisfied if $u$ is \SPR.\smallskip

We now state the main theorem of this section. Its proof is in
Appendix B.

\begin{theorem}\label{thm:3}
  If Assumptions 1, 2, and 3 are satisfied, then the generalized
  Pandora's rule maximizes expected utility.
\end{theorem}

\subsection{Application}
We conclude this section by applying Theorem~\ref{thm:3} to Example 1
introduced in Section 1, in which we presented the problem
of maximizing the expected value of
 \begin{equation}
    R(x_S^o) = \psi\left(\tsum_{i\in S}x_i^o\right)-\tsum_{i\in S}c_i
  \end{equation}
  where $\psi$ is a concave increasing function of the total value of
  prizes found. However, since we assume $x_i^o$ is either 0 or 1, we
  might also pose it as the problem of maximizing the expected value of
\[
R(x_S^o) = \tsum_{i=1}^{|S|}w_\ell x_\ell-\tsum_{i\in S}c_i
\]
where $w_i=\psi(i)-\psi(i-1)$.  Given that $k$ prizes have already
been found the reservation value $x_i^*$ is the least nonnegative $y$
such that
\[
\sum_{j=1}^{k}w_{i}+w_{k+1}y\geq -c_{i}+
\sum_{j=1}^{k}w_{i}+w_{k+1}y+p_i\Big((1-y)w_{k+1}+w_{k+2}y\Bigr)
\]
and hence
\[
x_i^*=\max \left\{ 0,\frac{w_{k+1}-c_{i}/p_{i}}{w_{k+1}-w_{k+2}}
\right\} .
\]
Thus the reservation values of the covered variables can be ranked greatest
to smallest in the order that their $c_i/p_i$ are ranked least to greatest,
independently of which other variables have been uncovered or not, and
the values taken by the uncovered variables. This means that ORD holds
and so by Theorem~\ref{thm:3} Pandora's rule is optimal.

While one could establish this fact by a hands-on and
fairly short proof, using an induction on $n$, the sufficient
conditions provided by Theorem~\ref{thm:3} are quicker to check.

\section{Conclusions}

Other researcher have also proved `negative results', similar to our
Theorem 1. For example, Banks and Sundaram \cite{Banks94} have shown
that the Gittins index cannot be extended to bandits with switching
costs. Also, Gittins, et al.\ \cite[Chapter 3]{Git11} show that the
Gittins index theorem holds only when one makes assumptions of
infinite time horizon, constant exponential discounting and that only
one bandit is continued at a time. Example 1 could be solved by a
Pandora's rule because of the simplified form of $F_i$. In other
research index policies are sometimes found to be optimal by imposing
so-called `compatibility constraints' on the parameters of the
problem.

There are other interesting Pandora box problems that remain unsolved.
For example, it would be very interesting to address a version of
Weitzman's problem in which the prizes (offers) do not remain
permanently available. When a box is opened its prize must be taken
immediately or permanently lost. This problem is unlikely to have a
simple answer, except in very special cases.

\nocite{*}

\section*{Appendix}
\appendix

In Sections A--D of this appendix we prove results described in
Section 2, showing that if Pandora's rule is to be optimal
for all choices of costs and distributions, $\{(c_i,F_i),\,i\in N\}$,
then this places very severe restrictions on the admissible form of
$u$. Section F contains the proof of Theorem~\ref{thm:3}.

\section{Proof of Lemma 1.}

We wish to show that if Assumption 1 holds, and the
  Pandora's rule maximizes expected utility for all costs $c_{i}$ and
  distributions $F_{i}$ then $u$ must satisfy Assumption 2.

  In the following proof, as in that of Theorem 1 below, we will
    use examples in which $F_i$ is a degenerate distribution. We might
    rewrite the proof using continuous random variables instead, by
    making perturbations in which absolutely continuous distributions
    approximate our degenerate ones. But to give a proof using
    degenerate distributions is both simpler and stronger. It is
    stronger since any random variable with a degenerate distribution
    can be approximated (arbitrarily closely) by a continuous random
    variable (but not vice versa). So if we show that the assumption
    that Pandora's rule is optimal for all $c_i,F_i$, with the $F_i$
    assumed `degenerate', then 
    the same is true if `degenerate' is replaced by `absolutely
    continuous'. The implication would not be true the other way
    around.

\begin{proof} 
Consider an arbitrary $S$, $x_S^o$, and $j,k\notin S$,
  with $j\neq k$, and numbers $x_j^o\leq \xu_k^o<\xo_k^o$.  Suppose
  there were a violation of Assumption 2 of the form
\[
u(x_S^o,x_j^o,\xL)-u(x_S^o,\xL)> u(x_S^o,x_j^o,\overline{x}_k^o)
-u(x_S^o,\overline{x}_k^o).
\]
Notice that we could not have the opposite strict inequality, by
Assumption 1 (submodularity). By Assumption 1, we can
increase $x_j^o$ to $\xL$ and the same inequality will hold.  This
implies that there exists $\epsilon>0$ such that
\begin{equation} 
  u(x_S^o,\xL,\xL)-u(x_S^o,\xL)> u(x_S^o,\xL,\overline{x}_k^o) -u(x_S^o,\overline{x}_k^o)+\epsilon.\label{eq:5a}
\end{equation} 

We now show that if \eqref{eq:5a} is true then Pandora's rule cannot
be optimal for all $(c_i,F_i,\ i\in N)$.  To this end, suppose $F_j$
and $F_k$ are degenerate, with $x_j^o=\xu_k^o$ and $x_k^o=\xo_k^o$
with probability 1.
Let
\begin{align*}
c_j&=u(x_S^o,\xL,\xL)-u(x_S^o,\xL)-\epsilon,\quad
c_k=u(x_{S}^{o},\xL,\overline{x}_{k}^{o})-u(x_{S}^{o},\xL).
\end{align*}
The reservation price of $x_j$ is the least nonnegative $y$ such that
\begin{equation}
  \label{eq:23}
c_j=   u(x_S^o,\xL,\xL)-u(x_S^o,\xL)-\epsilon\geq 
u(x_S^o,\xL,y) -u(x_S^o,y).
\end{equation} 
Since \eqref{eq:23}  is false for $y=\xu_k^o$ we must have  $x_j^*>\xu_k^o$.

The reservation price of $x_k$ is the least nonnegative $y$ such that
 \begin{align}
   c_k=u(x_{S}^{o},\xL,\xo_k^o)-u(x_{S}^{o},\xL)  
\geq u(x_S^o,\overline{x}_k^o,y) -u(x_S^o,y).\label{eq:16}
\end{align}
Suppose $y$ is such that \eqref{eq:23} holds, and therefore $y\geq
x_j^*>\xu_k^o$. Then \eqref{eq:16} also holds, since
\begin{align*}
 u(x_S^o,\overline{x}_k^o,y)&  -u(x_S^o,y)\\
&=u(x_S^o,\overline{x}_k^o,y)- u(x_S^o,\xu_k^o,y) +
  u(x_S^o,\xu_k^o,y)-u(x_S^o,y)\\
&\leq
u(x_S^o,\xo_k^o,y)-u(x_S^o,\xu_k^o,y)+u(x_S^o,\xL,\xL)-u(x_S^o,\xL)-\epsilon\\
&=c_k+[u(x_S^o,\xo_k^o,y)-u(x_S^o,\xu_k^o,y)]-[u(x_S^o,\xL,\xo_k^o)-u(x_{S}^{o},\xL,\xu_{k}^{o})]-\epsilon\\
&\leq c_k,
\end{align*}
where the first inequality is by \eqref{eq:23} and the second
inequality is by using Assumption 1 (submodularity) to see that since
$y\geq\xu_k^o$ the first square-bracketed term is no greater than the
second.

From this it follows that $x_k^*\leq x_j^*$.  Thus, according to
Pandora's rule, it would be optimal to next uncover $x_j$.

However, the payoff obtained by uncovering $x_k$ first and then
stopping search is strictly greater than the payoff obtained by
uncovering $x_j$ first and then stopping search if
\begin{equation}
  \label{eq:6a}
  u(x_{S}^{o},\xo_{k}^{o})-c_{k}>u(x_{S}^{o},\xL)-c_{j}.
\end{equation}
On substituting for $c_j$ and $c_k$ we find that \eqref{eq:6a} is the same as
(\ref{eq:5a}).

Note also that the right-hand side of \eqref{eq:6a} is nonnegative since
\begin{align*}
  u(x_{S}^{o},\xL)-c_{j}
&=u(x_{S}^{o},\xL)-[u(x_{S}^{o},\xL,\xL)-u(x_{S}^{o},\xL)]+\epsilon\\
&\geq u(x_{S}^{o},\xL)-[u(x_{S}^{o},\xL,0)-u(x_{S}^{o},0)]+\epsilon\\
&=u(x_{S}^{o})+\epsilon,
\end{align*}
where the inequality is by Assumption 1 (submodularity).  The payoff obtained
by uncovering $x_k$ first and then stopping search is also strictly
greater than the payoff of uncovering both if
\begin{equation}
  \label{eq:7a}
u(x_{S}^{o},\overline{x}_{k}^{o})-c_k>
u(x_{S}^{o},\xL,\overline{x}_{k}^{o})-c_{j}-c_{k}.
\end{equation}
On substituting for $c_j$ we find that \eqref{eq:7a} is also the same
as (\ref{eq:5a}).

Assume the parameters of all uncovered variables,  apart
from $x_j$ and $x_k$, are such that they should certainly stay
covered. For example, each such uncovered $x_i$ might have $x_i^o=0$ and
$c_i>0$.
We have argued that uncovering $x_k$ first and then stopping search is
strictly better than uncovering $x_j$ first and then either stopping
search or uncovering $x_k$.  As $x_k^*\leq x_j^*$, Pandora's rule
dictates that it is optimal to uncover variable $x_{j}$
first. As this is false, we must conclude that if Pandora's rule
is optimal then \eqref{eq:5a} must be false, and thus there can be no
violation to Assumption 2.
\end{proof}

\section{Proof of Lemma 2.}

\begin{proof}
 Define 
\begin{align*}
g_{\ell}(x)&=u(\underbrace{x,\dotsc,x}_{\ell\text{ times}})\\
f_{\ell}(x)&=g_{\ell}(x)-g_{\ell-1}(x)
\end{align*}
Then for $S=\{1,\dotsc,k\}$ and $x_S^o=(x_i,i\in S)$,
\begin{align}
  u(x_{S}^{o})&=
u(x_{1},\dotsc,x_{k})\nonumber\\
&=u(x_{1},\dotsc,x_{k})-u(x_{1},\dotsc,x_{k-1},0)+u(x_{1},\dotsc,x_{k-1})\label{eq:10a}\\
&=u(\underbrace{x_k,\dotsc,x_{k}}_{k\text{ times}})-u(\underbrace{x_{k},\dotsc,x_{k}}_{k-1\text{ times}},0)+u(x_{1},\dotsc,x_{k-1})\label{eq:10}\\
&=g_{k}(x_{k})-g_{k-1}(x_{k})+u(x_{1},\dotsc,x_{k-1})\nonumber\\
&=f_{k}(x_{k})+u(x_{1},\dotsc,x_{k-1})\nonumber\\
&=\dsum\limits_{\ell=1}^{k}f_{\ell}(x_{\ell}).\nonumber
\end{align} 
where \eqref{eq:10} follows from \eqref{eq:10a} by repeated
application of Assumption 2.  This proves (a).

For (b), the fact that $f_{\ell}(x)$ is a decreasing function of
$\ell$ follows from Assumption 1 (submodularity). The fact that
$f_{\ell}(x)$ is increasing in $x$ can be seen by taking
$x< x'$, and observing that
\[
g_{\ell}(x)-g_{\ell-1}(x)
=u(x,\underbrace{x',\dotsc,x'}_{\ell-1\text{ times}})
-u(\underbrace{x',\dotsc,x'}_{\ell-1\text{ times}}) 
\leq  g_{\ell}(x')-g_{\ell-1}(x'),
\]
where the equality is by Assumption 2, and the inequality
is by Assumption 1.

Finally, for (c), we note that if $x<x'$,
  \begin{align*}
    f_\ell(x)-f_{\ell+1}(x)&
=[u(\underbrace{x,\dotsc,x}_{\ell\text{ times}})
-u(\underbrace{x,\dotsc,x}_{\ell-1\text{ times}})]
-[u(\underbrace{x,\dotsc,x}_{\ell+1\text{ times}})
-u(\underbrace{x,\dotsc,x}_{\ell\text{ times}})]\\
&=[u(\underbrace{x',\dotsc,x'}_{\ell-1\text{ times}},x)
-u(\underbrace{x',\dotsc,x'}_{\ell-1\text{ times}})]
-[u(\underbrace{x',\dotsc,x'}_{\ell\text{ times}},x)
-u(\underbrace{x',\dotsc,x'}_{\ell\text{ times}})]\\
&\leq [u(\underbrace{x',\dotsc,x'}_{\ell\text{ times}})
-u(\underbrace{x',\dotsc,x'}_{\ell-1\text{ times}})]
-[u(\underbrace{x',\dotsc,x'}_{\ell+1\text{ times}})
-u(\underbrace{x',\dotsc,x'}_{\ell\text{ times}})]\\
&= f_\ell(x')-f_{\ell+1}(x').
  \end{align*}
  The second line is by Assumption 2, and the third line by Assumption
  1 (submodularity).
\end{proof}

\section{Proof of Theorem ~\ref{thm:1}}

We complete the proof of Theorem 1. As in the proof of Lemma 1 it is
convenient to use degenerate distributions in constructing
counterexamples.
  \begin{proof}[Proof of Theorem~\ref{thm:1}]
{\color{black} To complete the proof begun in Section 2 it
  remains to show that $u$ has form of (6), i.e.\ that in Lemma 2 we can put $f_{2}=\cdots=f_{n}$, where we have 
established in Lemma 2 that $f_{1}\geq f_{2}\geq
    \cdots\geq f_{n}$ and $f_{2}(0)=\cdots=f_{n}(0)=0$.  
  We do this for $f_2=f_3$.
      The proof of $f_\ell(x_{0})=f_{\ell+1}(x_{0})$, $\ell>2$ follows
      by examining an instance in which the first $\ell-2$ variables to be
      uncovered are ones with $c_i=0$ (and reservation values
      $\infty$) and their uncovered values are greater than any values
      that can be  found  amongst the variables which
      remain uncovered at that point. \smallskip

      (i)  Assume $f_3\neq 0$ and that
      there exists $x_0$ such that $f_1(x_0)\geq f_2(x_0)>f_3(x_0)>
      0$.   Consider three variables, $x_1$,
      $x_2$ and $x_3$ with the same degenerate distribution, having
      $x_i^o=x_0$, with probability 1, $i=1,2,3$.  Let costs
 be chosen so
\begin{equation}
c_3=0\leq c_1 < f_3(x_0)<c_2<f_2(x_0).
\label{7aa}
\end{equation}
We proceed to show the generalized Pandora's rule cannot be optimal.

Firstly, it follows from \eqref{7aa} that for all $y$ we have
$u(y)<-c_i+Eu(y,x_i^o)$.  Hence initially, when $S=\varnothing$, all
three variables have reservation value $\infty$.}
{\color{black}
  So if Pandora's rule is optimal then it must be optimal to uncover
  any of them first.
  Suppose $x_2$ is uncovered first, and then $x_{3}$ (which still has
  reservation value $\infty$). It is now strictly
  best to uncover $x_1$ iff
\[
f_1(x_0)+f_2(x_0) < -c_1 + f_1(x_0)+f_2(x_0)+f_3(x_0)
\]
which is true because $c_1< f_3(x_0)$. The payoff is that
of uncovering all three variables.

Alternatively, if we uncover $x_1$ first, followed by $x_3$, it is now
strictly best not to uncover $x_2$, since $c_2> f_3(x_0)$.
The difference between the expected payoffs of the strategy which
uncovers $x_1,x_3$ and of that which uncovers $x_2,x_3,x_1$ is
\begin{align*}
  & \Bigl[-c_1-c_3+f_1(x_0)+f_2(x_0) \Bigr]
  -\Bigl[-c_1-c_3-c_2+f_1(x_0)+f_2(x_0)+f_3(x_0)\Bigr]
  =c_2-f_3(x_0),
\end{align*}
which is positive, whereas if Pandora's rule were optimal this
difference should be no greater than 0.

(ii) Now consider the special case in which $f_3=0$ and $x_0$ is such
that $f_1(x_0)\geq f_2(x_0)>f_2(x_0)=0$. Suppose $x_i$ is a variable
such that $x_i^o$ is equal to $0$ or $x_0$ with probabilities
$q_i=1-p_i$ and $p_i$.
  Consider the class of variables like this, for varying
$p_i$ and $c_i$. All have initial reservation value $\infty$. Suppose
a variable in this class is uncovered and reveals value
$x_0$. Subsequent to this, the reservation value of another variable
in the class is now the least $y$, with $y\leq x_0$, such that
\[
f_1(x_0)+f_2(y)\geq
-c_i+f_1(x_0)+p_if_2(x_0)+(1-p_i)f_2(y)
\]
i.e.\ the least $y$ such that $f_2(y)\geq -c_i/p_i+f_2(x_0)$. Suppose
$c_i/p_i$ is chosen just a but less than $f_2(x_0)$, in such a way that
the reservation value is positive. Recall that
$f_2(y)$ is nondecreasing and continuous. Since $c_i/p_i$ may differ
and $f_2\neq 0$, variables in this class may now have different
reservation values. 

So suppose we start with three variables in this class. We uncover one
and it takes value $x_0$. The other two now have positive reservation
values, the greatest of which is for the variable with least value of
$c_i/p_i$.  In following  Pandora's rule we may start by uncovering
any variable initially, and then continue by uncovering variables in
increasing order of $c_\ell/p_\ell$, until either two values of $x_0$
have been revealed or all three variables have been uncovered.

If we uncover the variables in the order $x_i$, $x_j$, $x_k$ then the
expected payoff is
\begin{align*}
-c_i &+ p_i \Bigl[f_1(x_0)-c_j+p_jf_2(x_0)+q_j[-c_k+p_kf_2(x_0)]\Bigr]\\
&+q_i\Bigl[-c_j+p_j[f_1(x_0)-c_k+p_kf_2(x_0)]+q_j[-c_k+p_kf_1(x_0)]\Bigr]\\
&=(c_k/p_k)p_ip_jp_k+\sigma,
\end{align*}
where $\sigma$ is an expression that is symmetric in $i,j,k$. So if
$c_i/p_i<c_j/p_j<c_k/p_k<f_2(x_0)$ then it is strictly better to begin
by uncovering $x_i$ or $x_j$, than to begin by uncovering $x_k$. Thus
optimality of  Pandora's rule is incompatible with
$f_2(x_0)>f_3(x_0)=0$.  \qedhere}
\end{proof}


\section{A special case of Theorem~\ref{thm:1}}  

Some readers might consider the proof of Theorem 1 to be slightly
unsatisfactory because it refers to variables whose reservation value
is $\infty$. We conjecture, but have not been able to prove that
Theorem 1 is true even if we restrict attention to variables with
finite reservation values, at least under some mild restrictions on
utility $u$.

We can, however, prove Theorem 1 for the special utility described in
Theorem 4. The proof has similarities to the proof of Theorem 1, and
is omitted.

\begin{theorem}
  Suppose $w_1\geq w_2\geq\cdots$ are given, and for any
  $x_S^o$ the utility is
\[
u(x_S^o)=\sum_{i=1}^{|S|}x_{i}w_i,
\]
where $x_1\geq x_2\geq \cdots\geq x_{|S|}$ are the ordered values of
the components of $x_S^o=(x_i^o:\, i\in S)$.  Then

(a) $u(\cdot)$ satisfies Assumptions 1 and 2;\medskip

(b) if Pandora's rule is optimal for all
$\{(c_i,F_i),\,i\in N\}$ then necessarily
$w_2=w_3=\cdots$.
\end{theorem}

\section{Proof of Theorem~\ref{thm:3}}

The proof of Theorem~\ref{thm:3} is by induction with respect to the
number of remaining, uncovered variables. Fix a set of already
uncovered variables $x^{o}$, and suppose that Pandora's rule
applies to cases in which there are fewer remaining variables. We
compare the payoff of uncovering first the variable $k$ with the
greatest reservation value $x_{k}^*$, followed by some (possibly
suboptimal) strategy, which will be specified later, to the payoff of
uncovering first a variable $\ell$, followed by an optimal
continuation strategy, and we will show that the former payoff is no
less than the latter payoff.

We need the following property of reservation values.
\begin{lemma}\label{lem:4}
  The reservation value $x_{k}^*(x^{o})$ is a (weakly) decreasing
  function of the number of uncovered variables. That is, for
  any $x^{o}=x_{S}^{o}$, $x_{j}^{o}$ and $x_{k}$ such that $j,k\notin
  S$ and $j\neq k$,%
  \[
  x_{k}^*(x^{o})\geq x_{k}^*(x^{o},x_{j}^{o})\text{.\medskip }
  \]
\end{lemma}
\begin{proof}
By Assumption 1 (submodularity), 
  \[
  \int
  u(x^{o},x_{k}^*,x_{k}^{o})\,dF_{k}(x_{k}^{o})-u(x^{o},x_{k}^*)\geq
  \int
  u(x^{o},x_{j}^{o},x_{k}^*,x_{k}^{o})\,dF_{k}(x_{k}^{o})-u(x^{o},x_{j}^{o},x_{k}^*)
  \] 
  for every value of $x_{k}^*$. Therefore, Lemma~\ref{lem:4} follows from
  \eqref{eq:20}, defining $x_{k}^*$. Indeed, $x_{k}^*$ is
  determined by the intersection (more precisely, the point most to
  the left) of the graph of a decreasing function
  \[
  \int u(x^{o},x_{k}^*,x_{k}^{o})\,dF_{k}(x_{k}^{o})-u(x^{o},x_{k}^*)
  \]
  of variable $x_{k}^*$ with the horizontal line with intercept
  $c_{k}$ (assuming the usual convention regarding $0$ and $+\infty
  $).
\end{proof}
Equipped with Lemma 3, we can now prove Theorem~\ref{thm:3}. 

\begin{proof}[Proof of Theorem~\ref{thm:3}]
If $x_{k}^*=0$,
and so all reservation values are equal to $0$, it is an optimal strategy to
stop search without uncovering any additional variable. Indeed, since

\[
c_{k}\geq \int [u(x^{o},x_{k}^*,x_{k}^{o})-u(x^{o},x_{k}^*)]\,dF_{k}(x_{k}^{o}) 
\]
for every $x_{k}^*$,

\[
c_{k}\geq \int [u(x^{o},0,x_{k}^{o})-u(x^{o},0)]\,dF_{k}(x_{k}^{o})=\int
[u(x^{o},x_{k}^{o})-u(x^{o})]\,dF_{k}(x_{k}^{o}). 
\]

Thus, it is no worse to stop search without uncovering any variable
than to stop search after uncovering variable $x_{k}$; or analogously,
than to stop search after uncovering any other variable. However, by
Lemma~\ref{lem:4} and the inductive hypothesis, after uncovering any
other variable, for all values of that variable, it is optimal to stop
search.

In what follows, we assume that $x_{k}^*>0$. Consider the strategy of
uncovering variable $x_{\ell}$ first. We now characterize an optimal
continuation strategy. It follows from the inductive assumption and
ORD that an optimal continuation strategy is to uncover variable
$x_{k}$ if $x_{k}^*(x^{o},x_{\ell}^{o})>0$ and to stop search if
$x_{k}^*(x^{o},x_{\ell}^{o})=0$.

For any $x_{\ell}^{o}<x_{k}^*(x^{o})$,%
\[
u(x^{o},x_{\ell}^{o})\,<-c_{k}+\int
u(x^{o},x_{\ell}^{o},x_{k}^{o})\,dF_{k}(x_{k}^{o}). 
\]
Since 
\[
u(x^{o},0,x_{\ell}^{o})=u(x^{o},x_{\ell}^{o})\text{ and }%
u(x^{o},0,x_{\ell}^{o},x_{k}^{o})=u(x^{o},x_{\ell}^{o},x_{k}^{o}), 
\]
by Assumption 1 (its first part) we have that 
\[
u(x^{o},y,x_{\ell}^{o})\,<-c_{k}+\int u(x^{o},x_{\ell}^{o},y,x_{k}^{o})\,dF_{k}(x_{k}^{o})
\]
for sufficiently small $y$, which means that $x_{k}^*(x^{o},x_{\ell}^{o})>0$.

For any $x_{\ell}^{o}\geq x_{k}^*(x^{o})$,

\[
u(x^{o},x_{\ell}^{o})\,\geq -c_{k}+\int
u(x^{o},x_{\ell}^{o},x_{k}^{o})\,dF_{k}(x_{k}^{o}). 
\]
By Assumption 1 (submodularity), 
\[
u(x^{o},y,x_{\ell}^{o},x_{k}^{o})-u(x^{o},y,x_{\ell}^{o})\leq u(x^{o},x_{\ell}^{o},x_{k}^{o})-u(x^{o},x_{\ell}^{o}) 
\]
for every $y$. Thus, 
\[
u(x^{o},y,x_{\ell}^{o})\,\geq -c_{k}+\int
u(x^{o},x_{\ell}^{o},y,x_{k}^{o})\,dF_{k}(x_{k}^{o}) 
\]
for every $y$, which means that $x_{k}^*(x^{o},x_{\ell}^{o})=0$.

Thus, it is an optimal continuation strategy (after uncovering
variable $x_{\ell}$) to uncover variable $x_{k}$ if
$x_{\ell}^{o}<x_{k}^*(x^{o})$ and to stop search if $x_{\ell}^{o}\geq
x_{k}^*(x^{o})$. We compare the strategy of uncovering variable
$x_{\ell}$ first (followed by this optimal continuation strategy) to
the strategy of uncovering variable $x_{k}$ first, followed by
stopping search if the realization of this variable exceeds
$x_{k}^*(x^{o})$, and uncovering variable $x_{\ell}$ next if the
realization of variable $x_{k}$ falls below $x_{k}^*(x^{o})$.  In what
follows, we let $x_k^*=x_{k}^*(x^o)$ and assume that
$x_{k}^*<+\infty$. If $x_{k}^*=+\infty$, the two strategies
obviously yield the same payoff.

The comparison of the two strategies yields Figure 1(a), where the
vertical and horizontal lines are at the level of $x_{\ell}=x_{k}^*$
and $x_{k}=x_{k}^*$, respectively. Within each cell the upper line is
the payoff under the policy of uncovering $x_\ell$ first and the lower
line is the payoff under the policy of uncovering $x_k$ first. In the
south west cell of the diagram, both strategies yield the same
continuation payoff, which is denoted by
$U(x^{o},x_{\ell}^{o},x_{k}^{o})$. This is the greatest continuation
payoff, contingent on the realizations
$x_{\ell}^{o},x_{k}^{o}<x_{k}^*=x_{k}^*(x^{o})$.

\begin{figure}[H]
\renewcommand{\arraystretch}{1.4}
\centering
  \begin{minipage}[b]{.5\linewidth}
    \centering
     \[
\hspace*{-.8in}   \begin{array}{c|*{2}{>{\PBS\centering$}m{\tmplength}<{$}|}}
      \multicolumn{1}{c}{}&
      \multicolumn{1}{>{\PBS\centering$}m{\tmplength}<{$}}{x_\ell^o<x_k^*}&
      \multicolumn{1}{>{\PBS\centering$}m{\tmplength}<{$}}{x_\ell^o\geq x_k^*}\\[4pt]
      \cline{2-3}
{x_k^o\geq x_k^*} &   u(x^o,x_\ell^o,x_k^o)-c_\ell- c_k & u(x^o,x_\ell^o)-c_\ell \\[6pt]
                &   u(x^o,x_k^o)-c_k     & u(x^o,x_k^o)-c_k\\
      \cline{2-3}
{x_k^o< x_k^*}    &  U(x^o,x_\ell^o,x_k^o)-c_\ell-c_k &  u(x^o,x_\ell^o)-c_\ell\\[6pt]
                &  U(x^o,x_\ell^o,x_k^o)-c_k- c_\ell &
                u(x^o,x_\ell^o,x_k^o)-c_k- c_\ell\\[2pt]
      \cline{2-3}
    \end{array}
    \] {\hspace{.1in}Figure 1(a)}
  \end{minipage}\medskip\\
\end{figure}

Figure 1(b) is obtained from Figure 1(a) by deleting common terms within
upper and lower rows of each cell.
\begin{figure}[H]
\renewcommand{\arraystretch}{1.4}
\centering
  \begin{minipage}[b]{.5\linewidth}
    \centering
    \[
\hspace*{-.8in}    \begin{array}{c|*{2}{>{\PBS\centering$}m{\tmplength}<{$}|}}
      \multicolumn{1}{c}{}&
      \multicolumn{1}{>{\PBS\centering$}m{\tmplength}<{$}}{x_\ell^o<x_k^*}&
      \multicolumn{1}{>{\PBS\centering$}m{\tmplength}<{$}}{x_\ell^o\geq x_k^*}\\
      \cline{2-3}
      {x_k^o\geq x_k^*} &   u(x^o,x_\ell^o,x_k^o)-c_\ell & u(x^o,x_\ell^o)-c_\ell \\[4pt]
                        &  u(x^o,x_k^o)     & u(x^o,x_k^o)-c_k\\[4pt]
      \cline{2-3}
      {x_k^o< x_k^*} &  0 &  u(x^o,x_\ell^o)\\[4pt] 
                   &  0 &  u(x^o,x_\ell^o,x_k^o)-c_k\\[4pt]
      \cline{2-3}
    \end{array}
    \] {\hspace{.1in}Figure 1(b)}
  \end{minipage}\medskip\\
\end{figure}

Figure 1(c) is obtained from 2(b) by recalling that 
\begin{align}
u(x^{o},x_{k}^{\ast })& =-c_{k}+\int u(x^{o},x_{k}^{\ast},x_{k}^{o})\,dF_{k}(x_{k}^o)  \label{D1} \\
u(x^{o},x_{k}^{\ast })& \geq -c_{\ell }+\int u(x^{o},x_{k}^{\ast },x_{\ell}^{o})\,dF_{\ell}(x_{\ell}^o).  \label{D2}
\end{align}
Since $0<x_{k}^{\ast }<\infty $, (\ref{D1}) holds. Inequality
(\ref{D2}) holds by the definition of $x_{\ell }^{\ast }$, Assumption
1 (submodularity), and the assumption that $x_{\ell }^{\ast }<
x_{k}^{\ast } $. Figure 1(c) is obtained from 2(b) by using (\ref{D1})
to replace $-c_{k} $ in the two right-column cells and using
(\ref{D2}) to replace $-c_{\ell }$ in the two top-row cells with
something that is no less.

\begin{figure}[H]
\renewcommand{\arraystretch}{1.4}
\centering
  \begin{minipage}[b]{.5\linewidth}
    \centering
    \[
\hspace*{-.8in}    \begin{array}{c|*{2}{>{\PBS\centering$}m{\tmplength}<{$}|}}
      \multicolumn{1}{c}{}&
      \multicolumn{1}{>{\PBS\centering$}m{\tmplength}<{$}}{x_\ell^o<x_k^*}&
      \multicolumn{1}{>{\PBS\centering$}m{\tmplength}<{$}}{x_\ell^o\geq x_k^*}\\
      \cline{2-3}
      {x_k^o\geq x_k^*} &   u(x^o,x_\ell^o,x_k^o) &
      u(x^o,x_\ell^o) \\
&-u(x^o,x_\ell^o,x_k^*)+u(x^o,x_k^*)& -u(x^o,x_\ell^o,x_k^*)+u(x^o,x_k^*)\\[6pt]
      &      & u(x^o,x_k^o)\\[-8pt]
      &  \raisebox{12pt}{$u(x^o,x_k^o)$}                 & - u(x^o,x_k^o,x_k^*)+u(x^o,x_k^*)\\
      \cline{2-3}
      {x_k^o< x_k^*} &  0 &  u(x^o,x_\ell^o)\\[4pt]
                   &   &  u(x^o,x_\ell^o,x_k^o)\\[-4pt]
                   &  \raisebox{12pt}{0}  & -u(x^o,x_k^o,x_k^*)+u(x^o,x_k^*)\\
      \cline{2-3}
    \end{array}
    \] {\hspace{.1in}Figure 1(c)}
  \end{minipage}
\end{figure}

Notice finally that the entries in the top and bottom row of each cell
of Figure 1(c) are equal due to Assumption 2. Indeed, in the north west cell
\[
u(x^{o},x_{\ell}^{o},x_{k}^{o})-u(x^{o},x_{k}^{o})
=u(x^{o},x_{\ell}^{o},x_{k}^*)-u(x^{o},x_{k}^*) 
\]
because $x_{k}^{o}$, $x_{k}^*\geq
x_{\ell}^{o} $ in that cell. In  the north east cell
\[
u(x^{o},x_{\ell}^{o},x_{k}^*)-u(x^{o},x_{\ell}^{o})
=u(x^{o},x_{k}^{o},x_{k}^*)-u(x^{o},x_{k}^{o}) 
\]
because $x_{\ell}^{o}$, $x_{k}^{o}\geq x_{k}^* $ in that cell. Finally, in
the south east cell,
\[
u(x^{o},x_{k}^{o},x_{k}^*)-u(x^{o},x_{k}^{*})
=u(x^{o},x_{\ell}^{o},x_{k}^o)-u(x^{o},x_{\ell}^{o}) 
\]
because $x_{k}^*$, $x_{\ell}^{o}$ $\geq
x_{k}^{o}$ in that cell.

Thus, for any given $x^{o}$, the strategy of uncovering the variable with
the greatest $x_{k}^*$ is no worse than the strategy of uncovering any
other variable. To complete the inductive proof, we need to show
that the strategy of uncovering the variable with the greatest $x_{k}^*$
is no worse than stopping search. 

First note that by Assumption 1 (continuity), we
have in the case that $0<x_k^*<\infty$
\begin{equation}
u(x^{o},x_{k}^*)=-c_{k}+\int
u(x^{o},x_{k}^*,x_{k}^{o})\,dF_{k}(x_{k}^{o}).
\label{1}
\end{equation}
So by (\ref{1}),
\[
c_{k}=\int [u(x^{o},x_{k}^*,x_{k}^{o})-u(x^{o},x_{k}^*)]\,dF_{k}(x_{k}^{o}), 
\]
which by Assumption 1 (submodularity) yields
\[
c_{k}\leq \int [u(x^{o},x_{k}^{o})-u(x^{o})]\,dF_{k}(x_{k}^{o}). 
\]
This implies that uncovering variable $x_{k}$ is no worse than stopping
search.
\end{proof}

\end{document}